\documentclass[11pt,a4paper]{scrartcl}

\usepackage[centertags]{amsmath} 
\usepackage{color} 
\usepackage{amsfonts}
\usepackage{amssymb}
\usepackage{amsthm}
\usepackage{epsfig} 
\usepackage{rotating}
\usepackage{psfrag}
\usepackage{caption}
\usepackage{booktabs}
\usepackage{enumitem}
\usepackage{pdflscape}

\newcommand{\R}{{\mathbb R}} 
\newcommand{\Nat}{\mathbb{N}} 

\newcommand{\X}{\mathcal{X}} 
\newcommand{\Y}{\mathcal{Y}} 
\newcommand{\U}{\mathcal{U}} 
\newcommand{\B}{\mathcal{B}} 
\newcommand{\E}{\mathcal{E}} 
\newcommand{\C}{\mathcal{C}} 
\newcommand{\D}{\mathcal{D}} 
\newcommand{\T}{\mathcal{T}} 
\newcommand{\Q}{\mathcal{Q}} 
\renewcommand{\S}{\mathcal{S}}

\newcommand{\interior}{\mathrm{int}}   
 
\newcommand{\lb}[1]{\underline{#1}} 
\newcommand{\ub}[1]{\overline{#1}}

\newtheorem{theorem}{Theorem}
\newtheorem{corollary}[theorem]{Corollary}
\newtheorem{lemma}[theorem]{Lemma}

\newtheorem{assumption}{Assumption}
\newtheorem{remark}{Remark}
\newtheorem{definition}{Definition}

\begin{document}

\begin{center}
{\LARGE
\bf
\textsf{On the computation of $\lambda$-contractive sets\\[2mm] for linear constrained systems}
}

\renewcommand{\thefootnote}{$\dagger$} 

\vspace{5mm}
{
Moritz Schulze Darup\footnotemark[1] and Mark Cannon\footnotemark[1]
}
\vspace{2mm}

  \footnotetext[1]{M. Schulze Darup and M. Cannon are with 
the Control Group, Department of Engineering Science,
        University of Oxford, Parks Road, Oxford OX1 3PJ, UK.
        E-mail: {\tt moritz.schulzedarup@rub.de}.}
\end{center}

\paragraph{Abstract.}
We present two theoretical results on the computation of $\lambda$-contractive sets for linear systems with state and input constraints. First, we show that it is possible to a priori compute a number of iterations that is sufficient to approximate the maximal $\lambda$-contractive set with a given precision using 1-step sets.  Second, based on the former result, we provide a procedure for choosing $\lambda$ so that the associated maximal $\lambda$-contractive set is guaranteed to approximate the maximal controlled invariant set with a given accuracy.


\paragraph{Preamble.}
This is a preprint of an article to appear in IEEE Transactions on Automatic control.

\section{Introduction  and Problem Statement}
\label{sec:Introduction}

The concept of $\lambda$-contraction is widely used in control theory (see~\cite{Blanchini2008} for an excellent overview).
For linear discrete-time systems 
\begin{equation}
\label{eq:linSys} 
x(k+1)=A\,x(k)+B\,u(k)
\end{equation}
with state and input constraints of the form
\begin{equation}
\label{eq:constraintsXU} 
x(k) \in \X  \,\,\, \text{and} \,\,\, u(k) \in \U  \quad \text{for every} \quad k \in \Nat,
\end{equation}
the construction, the application and many fundamental properties of
 $\lambda$-contractive sets are well-known (see, e.g., \cite{Blanchini2008,Blanchini1994}).
The computation of $\lambda$-contractive sets often builds on the repeated evaluation of the mapping
\begin{equation}
\label{eq:Q1Lambda}
\Q_1^\lambda(\D):=\{x \in \X \,|\, \exists u \in \U : Ax +Bu \in \lambda\,\D\}
\end{equation}
for a given set $\D\subset \R^n$  and a fixed $\lambda \in (0,1]$.
This leads to a sequence of sets defined according to $\Q_0^\lambda(\D)=\D$ and 
\begin{equation}
\label{eq:QkLambdaSequence}
\Q_{k+1}^\lambda(\D):=\Q_1^\lambda(\Q_k^\lambda(\D)) 
\end{equation}
for $k \in \Nat$. Sequences of the form~\eqref{eq:QkLambdaSequence} were first addressed for the special case $\lambda = 1$ (see \cite{Bertsekas1972,Cwikel1986,Gutman1987,Keerthi1987}). Later, the case of $\lambda \in (0,1)$, which is more relevant here, was considered (see, e.g., \cite{Blanchini1994}).
For the special choice $\D=\X$, it is well-known that~\eqref{eq:QkLambdaSequence} results in a sequence of nested sets that approximate the maximal $\lambda$-contractive set $\C_{\max}^\lambda$ (see Def.~\ref{def:lambdaContractive}) from outside, i.e.,
\begin{equation}
\label{eq:CLambdaMaxOuterApprox}
\C_{\max}^\lambda \subseteq \Q_{k+1}^\lambda(\X) \subseteq \Q_k^\lambda(\X)
\end{equation}
for every $k \in \Nat$ (see \cite[Thm. 3.1]{Blanchini1994}). 
Moreover, if  $\C_{\max}^\lambda$ is a C-set (see Def.~\ref{def:Cset}), which is the case if $\lambda$ is such that $0 \in \interior(\C_{\max}^\lambda)$ \cite[Rem.~4.1]{Blanchini1994},
the sequence $\{\Q_k^\lambda(\X)\}$ converges to $\C_{\max}^\lambda$
in the sense that, for every $\epsilon>0$, there exists a $k \in \Nat$ such that
\begin{equation}
\label{eq:QkLambda1PlusEpsilonCmaxLambda}
\Q_{k}^\lambda(\X) \subseteq (1+\epsilon)\,\C_{\max}^\lambda.
\end{equation}
While this characteristic is,  in principle, well-understood, there does (except for $\lambda=1$) not exist a method to a priori compute an upper bound for a suitable~$k$ such that~\eqref{eq:QkLambda1PlusEpsilonCmaxLambda} holds for a given $\epsilon$.
In this paper, we derive such a bound  by adapting and extending related results on the convergence of null-controllable sets given in \cite{Cwikel1986}.

The presented bound on $k$ allows the solution of another problem related to the computation of $\lambda$-contractive sets. It is well-known that $\lambda$-contractive sets can be used to approximate the maximal controlled invariant set $\C^1_{\max}$ arbitrarily closely (consider, e.g., \cite[Lem. 2.1 and Thm. 2.1]{Blanchini1996}). However, it is not clear how to a priori choose $\lambda \in (0,1)$ such that the relation
\begin{equation}
\label{eq:muC1MaxSubsetCLambdaMax}
\mu \, \C^1_{\max} \subseteq \C^\lambda_{\max}
\end{equation}
is guaranteed to hold for a given $\mu \in (0,1)$. We show how such a $\lambda$ can be computed \textit{before} evaluating (or approximating) $\C^1_{\max}$ or $ \C^\lambda_{\max}$.

\section{Notation and Preliminaries}

We begin by formalizing the notion of $\lambda$-contractive sets, controlled invariant sets, and C-sets. As a preparation, note that the scaling $\mu\,\C $ is understood as $\mu\,\C := \{ \,\mu\,x \,| \, x \in \C \}$ for any scalar $\mu >0$ and any set $\C \subset \R^n$.

\begin{definition}
\label{def:lambdaContractive}
Let $\lambda \in (0,1]$.
A set $\C \subseteq \X$ is called $\lambda$-contractive for~\eqref{eq:linSys} w.r.t.~\eqref{eq:constraintsXU}, if for every $x \in \C$ there exists $u \in \U$ such that $A\,x+B\,u \in \lambda\,\C$.
For the special case $\lambda=1$, a $1$-contractive set is also called controlled invariant.
For a given $\lambda$, the maximal $\lambda$-contractive set (for~\eqref{eq:linSys} w.r.t.~\eqref{eq:constraintsXU}), i.e., the union of all $\lambda$-contractive sets for~\eqref{eq:linSys} w.r.t.~\eqref{eq:constraintsXU}, is denoted by $\C_{\max}^\lambda$.
\end{definition}

\begin{definition}
\label{def:Cset}
A set $\C \subset \R^n$ is called C-set if it is convex and compact and
contains the origin as an interior point.
\end{definition}

For two given C-sets $\C,\D \subset\R^n$, we define the distance between the sets as in~\cite[Sect.~2]{Cwikel1986}. Specifically, let $\S:= \{ \xi \in \R^n \,|\, \|\xi\|_2 = 1 \}$ denote a hypersphere in $\R^n$ and let 
$
\rho(\xi,\C) := \sup \{ \, \mu>0 \, |\, \mu \,\xi \in \C\}
$
 for any $\xi \in \S$.
Then 
\begin{equation}
\label{eq:dCD}
d(\C,\D) := \sup_{\xi \in \S}  \,\left| \ln \left(\rho(\xi,\C) \right) - \ln\left(\rho(\xi,\D)\right) \right|
\end{equation}
provides a measure of the distance between $\C$ and $\D$.
In fact, it is straightforward to show that $d:\Y \times \Y  \rightarrow \R$ is a metric on the set $\Y$ of all C-sets in $\R^n$. In particular, we have $d(\C,\D)=d(\D,\C)\geq 0$ and $d(\C,\D) \leq d(\C,\E) +d(\D,\E)$ for all  $\C,\D,\E \in \Y$. 
Moreover, $d(\C,\D)=0$ if and only if $\C=\D$.
Now, according to the following lemma (which we prove in the appendix), evaluating the distance $d(\C,\D)$ allows one to check relations of the form~\eqref{eq:QkLambda1PlusEpsilonCmaxLambda}.

\begin{lemma}
\label{lem:dCDepsilon}
Let $\delta\geq0$ and let $\C,\D \!\subset\! \R^n$ be C-sets with $\C \subseteq\! \D$. Then
\begin{align}
\label{eq:dImplications}
\D &\subseteq \exp(\delta) \,\C  \!\!\!\!\! &\Longleftrightarrow& \qquad  d(\C,\D) \leq \delta \qquad \text{and} \\
\label{eq:dComputation}
d(\C,\D)&= \min_{\mu \in \R} \,\ln(\mu)  &\text{s.t.}\,\,& \qquad  \D \subseteq \mu \,\C .
\end{align}
\end{lemma}

Finally, we introduce the shorthand notation $\Nat_{[i,k]}:=\{j \in \Nat \,|\, i \leq j\leq k\}$ and we denote the smallest and largest singular values of a matrix $\Phi \in \R^{n \times l}$ by $\sigma_{\min}(\Phi)$ and $\sigma_{\max}(\Phi)$, respectively. We further stress that most of the results presented in this paper require the two following assumptions on the system matrices $A\in \R^{n \times n}$ and $B \in \R^{n \times m}$ and the constraint sets $\X \subset \R^n$ and $\U \subset \R^m$ to hold.

\begin{assumption}
\label{assum:AB}
The pair $(A,B)$ is controllable.
\end{assumption}

\begin{assumption}
\label{assum:XU}
The sets $\X$ and $\U$ are C-sets.
\end{assumption}

\section{Proof of Contraction}

The following theorem provides the key to prove the statements about $\lambda$-contractive sets mentioned in the introduction and summarized in Thms.~\ref{thm:QkLambda1PlusEpsilonCmaxLambda} and~\ref{thm:muC1MaxSubsetCLambdaMax} further below.
Theorem~\ref{thm:dQnCDLeqEtaDCD} states that the mapping $\Q_n^\lambda(\C)$, i.e., the set $\Q_k^\lambda(\C)$ for $k$ equal to the state space dimension $n$, is a contraction on $(\Y,d)$ for every choice $\lambda \in (0,1]$. Note that Thm.~\ref{thm:dQnCDLeqEtaDCD} is similar to but different from \cite[Thm. 2.3]{Cwikel1986}. 
The most important difference is that \cite[Thm. 2.3]{Cwikel1986} only applies to the special case $\lambda =1$. In this paper, however, we are especially interested in cases where $\lambda \in (0,1)$.
Moreover, \cite[Thm. 2.3]{Cwikel1986} requires $A$ to be invertible (see, for instance, \cite[Eq. (2.4)]{Cwikel1986}, the computation of $\delta$ in the proof of \cite[Thm. 2.3]{Cwikel1986}, or \cite[Assum. 2.8]{Gutman1987}). This technical restriction is not necessary to show Thm.~\ref{thm:dQnCDLeqEtaDCD}.

\begin{theorem}
\label{thm:dQnCDLeqEtaDCD}
Let $\lambda \in (0,1]$ and let Assums.~\ref{assum:AB} and~\ref{assum:XU} be satisfied. Then, there exists an $\eta \in [0,1)$,  depending only on the system matrices $A$ and $B$, the constraints $\X$ and $\U$, and the contraction~$\lambda$,  such that
\begin{equation}
\label{eq:dQnCDLeqEtaDCD}
d(\Q_n^\lambda(\C),\Q_n^\lambda(\D)) \leq \eta \, d(\C,\D),
\end{equation}
for all C-sets $\C,\D \subset \R^n$ with $\C \subseteq \D$.
\end{theorem}

The proof of Thm.~\ref{thm:dQnCDLeqEtaDCD} requires some preparation.
First note that $\rho(\xi,\cdot )$ and $d(\cdot ,\cdot )$ are well-defined only for C-sets. Thus, the following observation is elementary to analyze~\eqref{eq:dQnCDLeqEtaDCD}.

\begin{lemma}
\label{lem:QkLambdaCSet}
Let $\lambda \in (0,1]$, let $\C \subset \R^n$ be a C-set, and let Assum.~\ref{assum:XU} be satisfied.
 Then $\Q_k^\lambda(\C)$ is a C-set for every $k\in \Nat$.
\end{lemma}

The proof of Lem.~\ref{lem:QkLambdaCSet} is well-known. See for example \cite[Props. 3.1 and 3.2]{Blanchini1994} and \cite[Rem. 2.2]{Cwikel1986}.
Moreover, it is trivial to show (by induction) that the following relation holds.

\begin{lemma}
\label{lem:QkLambdaCD}
Let $\lambda \in (0,1]$,  let $\C,\D \subset \R^n$ be C-sets with $\C \subseteq \D$,  and let Assum.~\ref{assum:XU} be satisfied. Then 
$
\Q_k^\lambda(\C) \subseteq \Q_k^\lambda(\D)
$
for every $k\in \Nat$.
\end{lemma}

Finally, it will be useful to state (necessary and sufficient) conditions for a state $x \in \R^n$ to be contained in the set $\Q_{k}^\lambda(\C)$ (resp.~$\Q_{k+1}^\lambda(\C)$). This is done in the following lemma.

\begin{lemma}
\label{lem:conditionsXinQkLambda}
Let $\lambda \in (0,1]$,  $k \in \Nat$, 
let $\C\subset\R^n$ be a C-set, and let Assum.~\ref{assum:XU} be satisfied. 
Then $x\in \Q_{k+1}^\lambda(\C)$ if and only if there exist
 $u_0,\dots,u_{k} \in \U$ and $\gamma \in \C$ such that
 \begin{align}
\label{eq:xInQkLambdaCondX}
&A^{j} x + \sum_{i=0}^{j-1}  A^{j-1-i} B \,\lambda^i  u_i \in \lambda^j\,\X \quad \forall \, j \in \Nat_{[0,k]} \quad \text{and} \\ 
\label{eq:xInQkLambdaCondGamma}
&A^{k+1} x + \sum_{i=0}^{k}  A^{k-i} B \, \lambda^i   u_i = \lambda^{k+1} \gamma.
\end{align}
\end{lemma}

\begin{proof}
Consider any C-set $\T \subset \R^n$ and note that $x \in \Q_{1}^\lambda(\T)$ if and only if (i) $x \in \X$ and (ii) there exist $u \in \U$ and $\tau \in \T$ such that
$ \lambda \, \tau = A\,x + B \, u$. Consequently, $x\in \Q_{k+1}^\lambda(\C)= \Q_1^\lambda(\Q_k^\lambda(\C))$ if and only if (i) $x \in \X$ and (ii) there exist $u_0 \in \U$ and $q_k \in \Q_k^\lambda(\C)$ such that
$ \lambda \, q_k = A\,x + B \, u_0$.
By the same reasoning, $q_k \in \Q_k^\lambda(\C)$ if and only if (i) $q_k \in \X$ and (ii) there exist $u_1 \in \U$ and $q_{k-1} \in \Q_{k-1}^\lambda(\C)$ such that
$ \lambda \, q_{k-1} = A\,q_{k} + B \, u_1$. Finally, $q_{1} \in \Q_1^\lambda(\C)$ if and only if (i) $q_1 \in \X$ and (ii) there exist $u_k \in \U$ and $q_{0} \in \Q_{0}^\lambda(\C)=\C$ such that
$ \lambda \, q_{0} = A\,q_1 + B \, u_k$.
The conditions \eqref{eq:xInQkLambdaCondX} and~\eqref{eq:xInQkLambdaCondGamma} are obtained by collecting the 
the conditions (i) and (ii), and defining $\gamma:=q_0$. 
\end{proof}

Lemmas~\ref{lem:QkLambdaCSet}, \ref{lem:QkLambdaCD}, and \ref{lem:conditionsXinQkLambda} allow to prove Thm.~\ref{thm:dQnCDLeqEtaDCD}.

\begin{proof}[Proof of Thm.~\ref{thm:dQnCDLeqEtaDCD}]
The proof consists of three parts. In part (i), we  derive a useful upper bound for $d(\Q^\lambda_n(\C),\Q^\lambda_n(\D))$. In part (ii), we identify some special states $x^\ast$ that are guaranteed to be contained in $\Q^\lambda_n(\C)$. These states are finally used to construct an $\eta \in [0,1)$ satisfying~\eqref{eq:dQnCDLeqEtaDCD} in part (iii) of the proof.

Part (i). Consider any C-sets $\C,\D \subset \R^n$ with $\C\subseteq \D$, let $\delta:=d(\C,\D)$, and define $\mu:=\exp(\delta)$. Note that $\delta \geq 0$ and $\mu \geq 1$. Moreover,  we have $\D \subseteq \mu \,\C$ according to Lem.~\ref{lem:dCDepsilon}.
Based on this,  we obtain $\Q^\lambda_n(\C) \subseteq \Q^\lambda_n(\D) \subseteq \Q^\lambda_n(\mu\,\C)$ 
according to Lem.~\ref{lem:QkLambdaCD} and
\begin{equation}
\nonumber
\rho(\xi,\Q^\lambda_n(\C)) \leq \rho(\xi,\Q^\lambda_n(\D)) \leq \rho(\xi,\Q^\lambda_n(\mu\,\C)) 
\end{equation}
for every $\xi \in \S$ by definition of $\rho(\cdot,\cdot)$. Thus, a first upper bound for the distance $d(\Q^\lambda_n(\C),\Q^\lambda_n(\D))$ reads
\begin{equation}
\label{eq:firstOverestimationDQnCQnD}
d(\Q^\lambda_n(\C),\Q^\lambda_n(\D))  \leq \sup_{\xi \in \S} \, \ln \left(\frac{\rho(\xi,\Q^\lambda_n(\mu\,\C))}{\rho(\xi,\Q^\lambda_n(\C))} \right).
\end{equation}

Part (ii). In the following, let $\xi \in \S$ be arbitrary but fixed. Define
 $\varrho := \rho(\xi,\Q^\lambda_n(\mu\,\C))$ and observe that $\varrho>0$. Let $x := \varrho\,\xi$ and note 
$x  \in \Q^\lambda_n(\mu\,\C)$. Thus, according 
to Lem.~\ref{lem:conditionsXinQkLambda}, there exist $u_0,\dots,u_{n-1} \in \U$ and $\gamma \in \mu\,\C$ such that Eqs.~\eqref{eq:xInQkLambdaCondX} and~\eqref{eq:xInQkLambdaCondGamma} hold for $k=n$.
For the case addressed here, conditions~\eqref{eq:xInQkLambdaCondX} and~\eqref{eq:xInQkLambdaCondGamma}
can be rewritten as follows.
There exist $u_0,\dots,u_{n-1} \in \U$ and $\gamma^\ast \in \C$ such that
\begin{align}
\label{eq:rhoXHatJ}
&A^{j} \varrho\, \xi + \sum_{i=0}^{j-1}  A^{j-1-i} B \, \lambda^i u_i \in \lambda^j\,\X \quad \forall \, j \in \Nat_{[0,n-1]} \quad  \text{and}  \\
\label{eq:rhoXHatN} 
& A^{n} \varrho\, \xi  + \sum_{i=0}^{n-1}  A^{n-1-i} B \, \lambda^i  u_i = \mu\,\lambda^{n} \gamma^\ast.
\end{align}
We now prove there also exist $v_0,\dots,v_{n-1} \in \U$ and $\hat{\varrho}>0$ such~that
\begin{align}
\label{eq:rhoHatXHatJ}
& A^{j} \hat{\varrho}\,\xi  + \sum_{i=0}^{j-1}  A^{j-1-i} B \,\lambda^i v_i \in \lambda^j\,\X \quad \forall\,\, j \in \Nat_{[0,n-1]} \quad  \text{and} \\ 
\label{eq:rhoHatXHatN}
& A^{n} \hat{\varrho}\, \xi + \sum_{i=0}^{n-1}  A^{n-1-i} B \, \lambda^i  v_i = 0.
\end{align} 
As a preparation, let $\B^n(r)$ denote a ball in $\R^n$ of radius $r$  centered at the origin.
Then, since $\X$ and $\U$ are C-sets by Assum.~\ref{assum:XU}, there  exist $\ub{r}_x \geq \lb{r}_x>0 $
and $\lb{r}_u >0$  such that 
\begin{equation}
\label{eq:rUrX}
\B^n(\lb{r}_x) \subseteq \X \subseteq \B^n(\ub{r}_x) 
 \qquad \text{and} \qquad \B^m(\lb{r}_u) \subseteq \U,
\end{equation}
which obviously implies $\sup_{\xi \in \S} \rho(\xi,\X) \leq \ub{r}_x$.
In addition, let 
\begin{equation}
\label{eq:alpha}
\alpha := \max \left\{ 1, \max_{j\in \Nat_{[1,n]}} \| A^j \|_2 \right\},
\end{equation}
and define
\begin{equation}
\label{eq:matricesPhi}
\Phi_j:= ( A^{j-1} B , \dots, A^0 B) \in \R^{n \times j m} 
\end{equation}
and 
$\vartheta_j := (\lambda^0 v_0, \dots, \lambda^{j-1} v_{j-1})^T \in \R^{j m}$
for every $j \in  \Nat_{[1,n]}$. Note that $\alpha \geq 1$  and $\sigma_{\min} (\Phi_n)>0$ by construction and since $\Phi_n$ has full rank 
as a consequence of Assum.~\ref{assum:AB}. Consequently, the choice
\begin{equation}
\label{eq:choiceOfRho}
\hat{\varrho}= \frac{\lambda^{n-1}}{\alpha} \min \left\{ \frac{\lb{r}_x}{1 + \frac{\sigma_{\max} (\Phi_n)}{\sigma_{\min} ( \Phi_n)}},\, \lb{r_u} \,\sigma_{\min} ( \Phi_n) \right\}
\end{equation}  
implies $\hat{\varrho} > 0$.
Moreover, we have
\begin{equation}
\label{eq:rewriteWithMatrices}
A^{j}  \hat{\varrho}\, \xi  + \sum_{i=0}^{j-1} \lambda^i A^{j-1-i} B \, v_i  =  A^{j} \hat{\varrho}\, \xi + \Phi_j \, \vartheta_j
\end{equation}
for every $j \in  \Nat_{[1,n]}$ by construction.
We next show that, for the state $\hat{x}:=\hat{\varrho} \,\xi$, we can always compute $n$ inputs $v_0,\dots,v_{n-1} \in \U$ such that~\eqref{eq:rhoHatXHatN} holds. In fact,
since the 
 Moore-Penrose pseudoinverse $\Phi_n^{+} := \Phi_n^T (\Phi_n\,\Phi_n^T)^{-1}$ has full rank (again due to Assum.~\ref{assum:AB}),
$\vartheta_n =  - \Phi_n^{+} A^{n} \hat{\varrho} \,\xi$ results a suitable choice.
Clearly, the associated inputs $v_0,\dots,v_{n-1}$ satisfy~\eqref{eq:rhoHatXHatN}. Moreover, we have
\begin{equation}
\label{eq:thetaNBound}
\left\|\vartheta_n\right\|_2  \leq \hat{\varrho}\, \|\Phi_n^{+}\|_2 \,\|A^{n}\|_2 \,\|\xi\|_2 \leq \frac{\alpha\,\hat{\varrho}}{\sigma_{\min}(\Phi_n)} 
\end{equation}
and consequently
$$
\left\| (v_0,\dots,v_{n-1})^T \right\|_2 \leq  \frac{1}{\lambda^{n-1}} \left\|\vartheta_n\right\|_2   \leq \frac{\hat{\varrho}\,\alpha}{\lambda^{n-1}\,\sigma_{\min}(\Phi_n)}  \leq \lb{r}_u,
$$
where the last relation holds because of~\eqref{eq:choiceOfRho}. Thus, we obtain $v_i \in\B^m(\lb{r}_u) \subseteq  \U$ for every $i \in \Nat_{[0,n-1]}$ according to~\eqref{eq:rUrX}.
To show~\eqref{eq:rhoHatXHatJ}, first note that, for every $j \in \Nat_{[1,n-1]}$, we obtain 
 \begin{align}
\nonumber
\left\|\,
 A^{j} \hat{\varrho}\,\xi  + \Phi_j \, \vartheta_j
 \right\|_2 
 &\leq \alpha\,\hat{\varrho} + \|\Phi_j\|_2   \left\|\vartheta_j \right\|_2 \leq \alpha\,\hat{\varrho} + \|\Phi_n\|_2   \left\|\vartheta_n \right\|_2 \\
 \nonumber
 & \leq \alpha\,\hat{\varrho}\,\left( 1 + \frac{\sigma_{\max} (\Phi_n)}{\sigma_{\min} ( \Phi_n)} \right) \leq \lambda^{n-1} \, \lb{r}_x
\end{align}
according to Eqs.~\eqref{eq:choiceOfRho}--\eqref{eq:thetaNBound} and due to $\|\Phi_n\|_2=\sigma_{\max}(\Phi_n)$.
Thus, \eqref{eq:rhoHatXHatJ} holds for every $j  \in \Nat_{[1,n-1]}$ since
$$\B^n(\lambda^{n-1} \lb{r}_x) = \lambda^{n-1} \,\B^n(\lb{r}_x) \subseteq\lambda^{n-1} \X \subseteq \lambda^{j} \X$$ 
follows from~\eqref{eq:rUrX}.
Clearly, \eqref{eq:rhoHatXHatJ} also holds for $j=0$ since $\hat{\varrho}\, \xi \in \B^n(\lambda^{n-1} \lb{r}_x) \subseteq \lambda^0 \X$ is ensured by~\eqref{eq:choiceOfRho}.
Now, combining the results from~\eqref{eq:rhoXHatJ}--\eqref{eq:rhoXHatN} and~\eqref{eq:rhoHatXHatJ}--\eqref{eq:rhoHatXHatN}, we can show that there exist
$w_0,\dots,w_{n-1} \in \U$ and $\varrho^\ast>0$ such that
\begin{align}
\label{eq:rhoAstXHatJ}
&A^{j} \varrho^\ast  \xi  + \sum_{i=0}^{j-1}  A^{j-1-i} B \,\lambda^i w_i \in \lambda^j\,\X \quad \forall\,\, j \in \Nat_{[0,n-1]}\quad  \text{and}\\ 
\label{eq:rhoAstXHatN}
& A^{n} \varrho^\ast  \xi  + \sum_{i=0}^{n-1}  A^{n-1-i} B \, \lambda^i w_i = \lambda^{n}\, \gamma^\ast.
\end{align}
In fact, the choices
\begin{equation}
\label{eq:rhoAst}
w_i := \frac{1}{\mu}\,u_i + \frac{\mu-1}{\mu}\,v_i \quad \text{and} \quad \varrho^\ast:= \frac{1}{\mu}\,\varrho + \frac{\mu-1}{\mu}\,\hat{\varrho},
\vspace{-1mm}
\end{equation}
which satisfy $w_i \in \U$ and $\varrho^\ast>0$, allow to rewrite~\eqref{eq:rhoAstXHatJ} as 
$$
A^{j} \varrho \,\xi + \!\sum_{i=0}^{j-1}  A^{j-1-i} B \,\lambda^i u_i  + (\mu-1)\Big( A^{j} \hat{\varrho}\,\xi+ \sum_{i=0}^{j-1}  A^{j-1-i} B \, \lambda^iv_i \Big) 
 \in \mu\, \lambda^j\,\X, 
$$
which holds according to~\eqref{eq:rhoXHatJ} and~\eqref{eq:rhoHatXHatJ}.
Analogously,~\eqref{eq:rhoAstXHatN} can be proven using Eqs.~\eqref{eq:rhoXHatN}, \eqref{eq:rhoHatXHatN}, and~\eqref{eq:rhoAst}. Thus, due to~\eqref{eq:rhoAstXHatJ}--\eqref{eq:rhoAstXHatN} and $\gamma^\ast\in \C$, we find $x^\ast:=\varrho^\ast \xi \in \Q^\lambda_n(\C)$ according to Lem.~\ref{lem:conditionsXinQkLambda}.

Part (iii). Clearly, $\varrho^\ast \xi \in \Q^\lambda_n(\C)$ implies $\varrho^\ast \leq \rho(\xi,\Q^\lambda_n(\C))$. Since $\xi \in \S$ was arbitrary, we obtain 
\begin{equation}
\label{eq:underestimationRhoQnC}
 \frac{\rho(\xi,\Q^\lambda_n(\mu\,\C)) + (\mu-1)\,\hat{\varrho}}{\mu}  \leq \rho(\xi,\Q^\lambda_n(\C))
 \end{equation}
for every $\xi \in \S$ according to~\eqref{eq:rhoAst} and by definition of $\varrho$. In addition, since $\Q^\lambda_n(\mu\,\C) \subseteq \X$ and due to~\eqref{eq:rUrX}, we have
\begin{equation}
\label{eq:overestimationSupRho}
\sup_{ \xi \in \S} \rho(\xi,\Q^\lambda_n(\mu\,\C)) \leq \sup_{\xi \in \S} \rho(\xi,\X) \leq \ub{r}_x.
\vspace{-1mm}
\end{equation}
We next define $\gamma\!:=\!\frac{\hat{\varrho}}{\ub{r}_x}$ and note that $\gamma \!\in \!(0,0.5]$ since $\hat{\varrho} \leq \frac{\lb{r}_x}{2}$ according to~\eqref{eq:choiceOfRho} (due to $\alpha \geq 1$,  $\sigma_{\max} (\Phi_n) \geq \sigma_{\min} ( \Phi_n)>0$, and  $\lambda \in (0,1]$).
Now, combining~\eqref{eq:underestimationRhoQnC} and \eqref{eq:overestimationSupRho}, we  infer 
\vspace{-1mm}
\begin{align}
\nonumber
\sup_{\xi \in \S} \,\ln \left(\frac{\rho(\xi,\Q^\lambda_n(\mu\,\C))}{\rho(\xi,\Q^\lambda_n(\C))} \right) &\leq  \,\sup_{\xi \in \S} \,\ln \left( \frac{\mu}{1 + \frac{(\mu-1)\,\hat{\varrho}}{\rho(\xi,\Q^\lambda_n(\mu\,\C))}} \right) \leq  \ln \left( \frac{\mu}{1 + \frac{(\mu-1)\,\hat{\varrho}}{\ub{r}_x}} \right) \\
\nonumber
& = \ln(\mu) -  \ln \left(1 + (\mu-1)\,\gamma \right)\leq \ln(\mu) - \ln(\mu^\gamma) \\
\label{eq:finalObservation} 
&= (1-\gamma) \, \ln(\mu), 
\end{align}
where the last inequality holds since $\mu^\gamma$ underestimates $1 + (\mu-1) \,\gamma$  (given that $\mu\geq1$ and $\gamma \in (0,0.5]$).
Finally, taking $\ln(\mu)=\delta=d(\C,\D)$ and~\eqref{eq:firstOverestimationDQnCQnD}  into account, it is easy to see that~\eqref{eq:finalObservation} 
 proves~\eqref{eq:dQnCDLeqEtaDCD} for the choice  $\eta: = 1 - \gamma=1-\frac{\hat{\varrho}}{\ub{r}_x}$.
In fact, with regard to relations~\eqref{eq:rUrX}--\eqref{eq:choiceOfRho}, this choice of $\eta$ only depends on the system matrices $A$ and $B$, the constraints $\X$ and $\U$, and the contraction~$\lambda$.  Moreover, we have $\eta \in [0.5,1) \subset [0,1)$ due to $\gamma \in (0,0.5]$.
 \end{proof}

Theorem~\ref{thm:dQnCDLeqEtaDCD} establishes the contraction of $\Q_n^\lambda(\C)$. Using similar arguments (but omitting the variations~\eqref{eq:rhoHatXHatJ} and~\eqref{eq:rhoHatXHatN} which require controllability), it is easy to prove the following weaker relation, which however holds for every $k \in \Nat$.

\begin{corollary}
\label{cor:dQkLambdaLeqD}
Let $\lambda \in (0,1]$ and let Assum.~\ref{assum:XU}  be satisfied. Then 
\begin{equation}
\label{eq:weakerRelation}
d(\Q_k^\lambda(\C),\Q_k^\lambda(\D)) \leq  d(\C,\D)
\end{equation}
 for every $k \in \Nat$ and all C-sets $\C,\D \subset \R^n$ with $\C \subseteq \D$.
\end{corollary}

Theorem~\ref{thm:dQnCDLeqEtaDCD} and Cor.~\ref{cor:dQkLambdaLeqD} lead to Lem.~\ref{lem:appropriateKCD}, which will be instrumental to prove Thms.~\ref{thm:QkLambda1PlusEpsilonCmaxLambda} and~\ref{thm:muC1MaxSubsetCLambdaMax} in the next section.
As a preparation, we state the following corollary, which summarizes the choice of a suitable contraction factor $\eta$ according to the proof of Thm.~\ref{thm:dQnCDLeqEtaDCD}. 

\begin{corollary}
\label{cor:choiceOfEta}
Let $\lambda \in (0,1]$ and let Assums.~\ref{assum:AB} and~\ref{assum:XU} be satisfied. Moreover, let $\lb{r}_x$, $\ub{r}_x$, and $\lb{r}_u$ with $0<\lb{r}_x \leq  \ub{r}_x$ and $0<\lb{r}_u$ be such that~\eqref{eq:rUrX} holds, define $\alpha$ and $\Phi_n$ according to~\eqref{eq:alpha} and~\eqref{eq:matricesPhi}, respectively,  and choose  $\hat{\varrho}$ as in~\eqref{eq:choiceOfRho}.
 Then, $\eta = 1 - \frac{\hat{\varrho}}{\ub{r}_x}$ is such that~\eqref{eq:dQnCDLeqEtaDCD} holds for all C-sets $\C,\D \subset \R^n$ with $\C \subseteq \D$.
\end{corollary}

\begin{lemma}
\label{lem:appropriateKCD}
Let $\lambda \in (0,1]$ and  $\delta>0$,  let $\C,\D \subset \R^n$ be C-sets with $\C \subseteq \D$, and let Assums.~\ref{assum:AB} and~\ref{assum:XU} be satisfied.
Choose $\eta$ according to Cor.~\ref{cor:choiceOfEta}.
 Then  $d(\Q_k^\lambda(\C), \Q_k^\lambda(\D)) \leq \delta$ 
 for every $k\in \Nat$ with
\begin{equation}
\label{eq:kGeg2Cases}
k \geq \left\{ \begin{array}{ll}
n\, \left\lceil \frac{\ln(\delta)-\ln(d(\C,\D))}{\ln(\eta)} \right\rceil & \text{if}\,\,d(\C,\D) > \delta, \\
0 & \text{otherwise}. 
\end{array}\right.
\end{equation}
\end{lemma}

\begin{proof} 
If $d(\C,\D)\leq \delta$, we find $d(\Q_k^\lambda(\C), \Q_k^\lambda(\D)) \leq \delta$ for every $k \in \Nat$ according to Cor.~\ref{cor:dQkLambdaLeqD}.
It remains to address the case $d(\C,\D)> \delta$.
Theorem~\ref{thm:dQnCDLeqEtaDCD} implies $d(\Q_{j n}^\lambda(\C), \Q_{j n}^\lambda(\D)) \leq \eta^{j} d(\C,\D)$ for every $j \in \Nat$.
Clearly, $\eta^{j} d(\C,\D) \leq \delta$ if 
$$
j \geq \frac{\ln(\delta)-\ln(d(\C,\D)}{\ln(\eta)}.
$$ 
Taking $j \in \Nat$ into account leads to~\eqref{eq:kGeg2Cases}. 
  \end{proof}

\section{Implications}

In this section, we provide formal proofs for the new results on $\lambda$-contractive sets stated in the introduction and summarized in
Thms.~\ref{thm:QkLambda1PlusEpsilonCmaxLambda} and~\ref{thm:muC1MaxSubsetCLambdaMax}.

\begin{theorem}
\label{thm:QkLambda1PlusEpsilonCmaxLambda}
Let $\lambda \in (0,1]$ and $\epsilon>0$, let Assums.~\ref{assum:AB} and~\ref{assum:XU} be satisfied, and let $\C$ be any $\lambda$-contractive C-set.
Set $\D=\X$ and $\delta=\ln(1+\epsilon)$, choose $\eta$ according to Cor.~\ref{cor:choiceOfEta}, and let $k$ be such that~\eqref{eq:kGeg2Cases} holds.
Then
\begin{equation}
\label{eq:QkLambdaXSubsetQkLambdaC}
\Q_k^\lambda(\X) \subseteq (1+\epsilon)\, \Q_k^\lambda(\C) \subseteq (1+\epsilon)\, \C_{\max}^{\lambda}.
\end{equation}
\end{theorem}
\begin{proof}
Since $\C$ is $\lambda$-contractive, it is easy to show that $\Q_k^\lambda(\C)$ is $\lambda$-contractive for every $k \in \Nat$. Thus, we  have $\Q_k^\lambda(\C) \subseteq \C_{\max}^\lambda $ 
for every $k \in \Nat$, which proves the second relation in~\eqref{eq:QkLambdaXSubsetQkLambdaC}.
For a $k$ satisfying~\eqref{eq:kGeg2Cases}, we obtain 
$d(\Q_k^\lambda(\C), \Q_k^\lambda(\X)) \leq \delta$
according to Lem.~\ref{lem:appropriateKCD}.
Hence, the first relation in~\eqref{eq:QkLambdaXSubsetQkLambdaC} holds according to Lem.~\ref{lem:dCDepsilon} and due to the choice of $\delta$.
  \end{proof}

\begin{remark}
\label{rem:lambdaCSet}
As mentioned in the introduction, existence of a $k \in \Nat$ satisfying~\eqref{eq:QkLambda1PlusEpsilonCmaxLambda} for given $\lambda$ and $\epsilon$ requires $\C_{\max}^\lambda$ to be a C-set.
A necessary and sufficient condition for $\C_{\max}^\lambda$ to be a C-set is the existence of some $\lambda$-contractive C-set $\C$.
Clearly, the explicit knowledge of such a set $\C$, as required in Thm.~\ref{thm:QkLambda1PlusEpsilonCmaxLambda},
is more restrictive than the fundamental assumption that $\C_{\max}^\lambda$ is a C-set. However, there exist a number of procedures to identify (small) $\lambda$-contractive C-sets.
Assume, for example, there exist $ K \in \R^{m \times n}$ and a positive definite matrix $P \in \R^{n \times n}$ such that
$(A+BK)^T P (A+BK)\preceq \lambda^2 P$  and such that $A+BK$ is Schur stable.
Now choose any $\beta>0$ such that
$\C=\{ x \in \R^n \, | \, x^T P x \leq \beta \} \subseteq \{ x \in \X \,|\, K \,x \in \U\}$.
Then, $\C$ is a $\lambda$-contractive ellipsoid and thus a $\lambda$-contractive C-set (cf. \cite[Rem. 4.1]{Blanchini1994}). Similar procedures exist to compute $\lambda$-contractive polytopes (see, e.g., \cite[Sect. V]{Fiacchini2007}).
\end{remark}

Theorem~\ref{thm:muC1MaxSubsetCLambdaMax} further below addresses the suitable choice of $\lambda \in (0,1)$ to guarantee~\eqref{eq:muC1MaxSubsetCLambdaMax} for a given $\mu \in (0,1)$.
As a preparation, we provide the following two lemmas.

\begin{lemma}
\label{lem:QkLambda1}
Let $\lambda \in (0,1]$, let $\D\subset\R^n$ be a C-set, and let Assum.~\ref{assum:XU} be satisfied. Then  $\lambda^k\,\Q_k^1(\D) \subseteq \Q_k^\lambda(\D)$ for every $k\in \Nat$.
\end{lemma}

\begin{proof}
The relation holds with equality for $k=0$.
We prove the relation for $k>0$ by induction.
First note that, for any C-set $\T \subset \R^n$, $\Q_1^\lambda(\T)$ can be written as
\begin{equation}
\label{eq:Q1LambdaT}
\Q_1^\lambda(\T)=\lambda\,\{\tilde{x} \in \lambda^{-1}\X \,|\, \exists \tilde{u} \in \lambda^{-1}\,\U : A\tilde{x} +B \tilde{u} \in \T \}.
\end{equation}
Since $\lambda\leq 1$ implies $\lambda^{-1}\X \supseteq \X$ and $\lambda^{-1}\,\U \supseteq \U$, we obtain $\lambda\,\Q_1^1(\T) \subseteq \Q_1^\lambda(\T)$.
To show that $\lambda^{k}\,\Q_{k}^1(\D) \subseteq \Q_{k}^\lambda(\D)$ implies $\lambda^{k+1}\,\Q_{k+1}^1(\D) \subseteq \Q_{k+1}^\lambda(\D)$, first note that
\begin{equation}
\label{eq:QKPlusLambda}
\!\Q_{k+1}^\lambda(\D)\!=\!\Q_1^\lambda(\Q_k^\lambda(\D)) \!\supseteq\! \Q_1^\lambda(\lambda^k\,\Q_k^1(\D)) \!\supseteq\!\lambda\,\Q_1^1(\lambda^k\,\Q_k^1(\D)).\!\!\!\!
\end{equation}
Now, rewriting $\Q_1^1(\mu\,\T)$ in the style of~\eqref{eq:Q1LambdaT} for some $\mu \in (0,1]$, it is easy to show that $\mu\,\Q_1^1(\T) \subseteq \Q_1^1(\mu\,\T)$. We consequently find 
$\lambda^{k}\,\Q_1^1(\T) \subseteq \Q_1^1(\lambda^k\,\T)$. Taking~\eqref{eq:QKPlusLambda} into account, we finally infer
$$
\Q_{k+1}^\lambda(\D) \supseteq \lambda^{k+1}\,\Q_1^1(\Q_k^1(\D)) = \lambda^{k+1}\,\Q_{k+1}^1(\D),
$$
which completes the proof.  
  \end{proof}

\begin{lemma}
\label{lem:muLambdaK}
Let $\mu \in (0,1)$ and $\lambda^\ast \in (0,1)$,  let Assums.~\ref{assum:AB} and~\ref{assum:XU} be satisfied,  let $\C$ be a $\lambda^\ast$-contractive C-set, and set $\epsilon = \frac{1-\mu}{2\,\mu}$.
Then, there exist $\lambda \in [\lambda^\ast,1)$ and $k\in \Nat$
 such that~\eqref{eq:QkLambdaXSubsetQkLambdaC} holds and such that
\begin{equation}
\label{eq:muLambdaK}
1+\mu \leq 2\,\lambda^k.
\end{equation} 
\end{lemma}

\begin{proof}
Set $\lambda = \lambda^\ast$,  $\D=\X$, and $\delta = \ln(1+\epsilon)$ and let $\lb{r}_x$, $\ub{r}_x$, $\lb{r}_u$, $\alpha$ and $\Phi_n$ be as in Cor.~\ref{cor:choiceOfEta}.
Choose  $\hat{\varrho}$ as in~\eqref{eq:choiceOfRho}, set $\eta=1 - \frac{\hat{\varrho}}{\ub{r}_x}$,
and pick any $k \in \Nat$ that satisfies~\eqref{eq:kGeg2Cases}.
Then, relation~\eqref{eq:QkLambdaXSubsetQkLambdaC} holds
according to Thm.~\ref{thm:QkLambda1PlusEpsilonCmaxLambda}.
However, we either have (i)  $1+\mu \leq 2\,(\lambda^\ast)^k$ or (ii) $1+\mu > 2\,(\lambda^\ast)^k$. Case (i) immediately finishes the proof.

In contrast, if case (ii) applies, first note that we have $k>0$ (since $1+\mu >2 (\lambda^\ast)^0 = 2$ contradicts $\mu<1$). Now, 
 compute
\begin{equation}
\label{eq:lambdaUpdate}
\lambda = \exp \left( \frac{\ln(1+\mu)-\ln(2)}{k} \right)
\end{equation}
 and note that
$2 \lambda^{k} = 1+\mu$ and $\lambda \in (\lambda^\ast,1)$. Clearly, since $\lambda>\lambda^\ast$, the set $\C$ is also $\lambda$-contractive. Now, recompute $\hat{\varrho}$ according to~\eqref{eq:choiceOfRho} for the new value of $\lambda$ given by~\eqref{eq:lambdaUpdate}.
Note that the new $\hat{\varrho}$ is larger than the one that was obtained above with $\lambda=\lambda^\ast$.
Consequently, the recalculation of $\eta=1 - \frac{\hat{\varrho}}{\ub{r}_x}$ results in a smaller value than above.
Thus,  it is easy to see that 
$k$ as chosen above still satisfies~\eqref{eq:kGeg2Cases}.
This completes the proof, since \eqref{eq:QkLambdaXSubsetQkLambdaC} again holds 
according to Thm.~\ref{thm:QkLambda1PlusEpsilonCmaxLambda} and since~\eqref{eq:muLambdaK} is satisfied  by construction.
 \end{proof}

\begin{theorem}
\label{thm:muC1MaxSubsetCLambdaMax}
Let $\mu \in (0,1)$ and $\lambda^\ast \in (0,1)$,  let Assums.~\ref{assum:AB} and~\ref{assum:XU} be satisfied,  let $\C$ be a $\lambda^\ast$-contractive C-set, and set $\epsilon = \frac{1-\mu}{2\,\mu}$.
Assume $\lambda \in\! [\lambda^\ast,1)$ and $k\in\! \Nat$
are such that~\eqref{eq:QkLambdaXSubsetQkLambdaC} and~\eqref{eq:muLambdaK} hold.
Then
\begin{equation}
\label{eq:muCmax1SubsetQkLambdaCSubsetCmaxLambda}
\mu \,\C_{\max}^1 \subseteq  \Q_k^\lambda(\C) \subseteq \C_{\max}^\lambda.
\end{equation}
\end{theorem}

\begin{proof}
We have $\C^1_{\max} \subseteq \Q_k^1(\X)$ and $\lambda^k\Q_k^1(\X) \subseteq \Q_k^\lambda(\X)$ according to Eq.~\eqref{eq:CLambdaMaxOuterApprox} and Lem.~\ref{lem:QkLambda1}, respectively. Combining both relations and taking Eq.~\eqref{eq:QkLambdaXSubsetQkLambdaC} into account, yields
$$
\lambda^k\,\C^1_{\max} \subseteq \lambda^k\,\Q_k^1(\X) \subseteq \Q_k^\lambda(\X) \subseteq (1+\epsilon)\, \Q_k^\lambda(\C) \subseteq (1+\epsilon)\, \C_{\max}^{\lambda}.
$$
This proves~\eqref{eq:muCmax1SubsetQkLambdaCSubsetCmaxLambda} since we have
$$
\mu \,\C_{\max}^1 \subseteq \frac{2\lambda^k }{1+\mu} \,\mu \, \C_{\max}^1 = \frac{\lambda^k}{1+\epsilon} \, \C_{\max}^1
$$
due to~\eqref{eq:muLambdaK} and by definition of $\epsilon$, respectively.
  \end{proof}

\begin{remark}
\label{rem:ComputationLambdaSetTheory}
For the interpretation of Lem.~\ref{lem:muLambdaK} and Thm.~\ref{thm:muC1MaxSubsetCLambdaMax}, it is important to note that, for a given $\mu \in (0,1)$, suitable $\lambda \in [\lambda^\ast,1)$ and $k \in \Nat$ satisfying~\eqref{eq:QkLambdaXSubsetQkLambdaC} and~\eqref{eq:muLambdaK} can be computed without evaluating the sets $\Q_k^\lambda(\X)$, $\Q_k^\lambda(\C)$, or $\C_{\max}^\lambda$ in~\eqref{eq:QkLambdaXSubsetQkLambdaC}. In fact, 
we only require the computation of (i) a $\lambda^\ast$-contractive C-set $\C$, (ii) the distance $d(\C,\X)$ (e.g., according to~\eqref{eq:dComputation}), and (iii) $\eta$ as in Cor.~\ref{cor:choiceOfEta}. Then, suitable $\lambda$ and $k$ can be calculated according to the proof of Lem.~\ref{lem:muLambdaK}.
\end{remark}

\begin{remark}
\label{rem:ComputationLambdaSetPractice}
For practical applications,  having the guarantee that~\eqref{eq:muCmax1SubsetQkLambdaCSubsetCmaxLambda} holds without actually knowing (an approximation of) $\C_{\max}^\lambda$ is usually useless.
Fortunately, Thm.~\ref{thm:muC1MaxSubsetCLambdaMax} implicitly provides two methods to compute $\lambda$-contractive sets that approximate the maximal controlled invariant set $\C_{\max}^1$ with a given accuracy $\mu \in (0,1)$, presupposing that a $\lambda^\ast$-contractive set $\C$ is known.
To see this, first note that \textit{any} $\lambda \in [\lambda^\ast,1)$ and $k \in \Nat$ satisfying \eqref{eq:QkLambdaXSubsetQkLambdaC} and~\eqref{eq:muLambdaK} for $\epsilon = \frac{1-\mu}{2\,\mu}$ result in a $\lambda$-contractive set $\T=\Q_k^\lambda(\C)$ satisfying
$\mu \,\C_{\max}^1 \subseteq  \T$. 
Now, a suitable set $\T$ can be computed using two strategies.
First, we can compute $\lambda$ and $k$ according to the proof of Lem.~\ref{lem:muLambdaK} and simply evaluate $\Q_k^\lambda(\C)$ according to~\eqref{eq:QkLambdaSequence}.
This, however, might be numerically expensive since the computed $k$ is usually quite conservative.
Second, for $\lambda$ as above, we can compute $\Q_j^\lambda(\C)$ and $\Q_j^\lambda(\X)$ for increasing $j \in \Nat$ until $\Q_j^\lambda(\X) \subseteq (1+\epsilon)\, \Q_j^\lambda(\C)$ is observed for $j=k^\ast$. We obviously have $k^\ast \leq k$ by construction and thus 
$1+\mu \leq 2\,\lambda^k \leq 2\, \lambda^{k^\ast}$.
Consequently, \eqref{eq:QkLambdaXSubsetQkLambdaC} and~\eqref{eq:muLambdaK} hold and $\T=\Q_{k^\ast}^\lambda(\C)$ approximates $\C_{\max}^1$ accurately. The second strategy is numerically attractive, since $k^\ast$ is usually significantly smaller than $k$ from the proof of Lem.~\ref{lem:muLambdaK} (see the example in Sect.~\ref{subsec:ExampleThm12}). 
\end{remark}

\section{Illustrative examples}
\label{sec:example}

We analyze three examples to discuss the uses and limitations of Thms.~\ref{thm:QkLambda1PlusEpsilonCmaxLambda} and \ref{thm:muC1MaxSubsetCLambdaMax} as well as Thm.~\ref{thm:dQnCDLeqEtaDCD}.
In particular, we show how to compute suitable $k$ and $\lambda$ such that~\eqref{eq:QkLambda1PlusEpsilonCmaxLambda} and~\eqref{eq:muC1MaxSubsetCLambdaMax} hold \textit{without} evaluating   $\Q_k^\lambda(\X)$, $\C_{\max}^\lambda$, or $\C_{\max}^1$.
These numbers are then compared to the optimal (i.e., smallest possible) choices of $k$ and $\lambda$. It is easy to see that the optimal choices \textit{require} the knowledge of the sets $\Q_k^\lambda(\X)$, $\C_{\max}^\lambda$, or $\C_{\max}^1$. Since explicit descriptions of these sets are usually not available for complex systems, we consider relatively simple examples in the following.
We stress, however,  that the techniques in Thms.~\ref{thm:QkLambda1PlusEpsilonCmaxLambda} and \ref{thm:muC1MaxSubsetCLambdaMax} can be applied to more complex systems provided the requirements in Rem.~\ref{rem:ComputationLambdaSetTheory} can be satisfied.
 
\subsection{Discussion of Theorem~\ref{thm:QkLambda1PlusEpsilonCmaxLambda}}
\label{subsec:ExThm9}

Theorem~\ref{thm:QkLambda1PlusEpsilonCmaxLambda} makes it possible to compute an iteration bound $k$ such that \eqref{eq:QkLambdaXSubsetQkLambdaC} holds. In the following, we compare the provided bound with the smallest $k$ satisfying \eqref{eq:QkLambdaXSubsetQkLambdaC} for a simple example. 

Consider system~\eqref{eq:linSys} with
$A=1.1\, I_n$ and $B=I_n$ and constraints $\X=[-10,10]^n$ and $\U=[-1,1]^n$ for an arbitrary $n \in \Nat$ with $n>0$, where $I_n$ denotes  the identity matrix  in $\R^{n \times n}$.
Obviously, the system can be resolved into $n$ independent systems of dimension one. Nevertheless, the conglomerated system is useful to analyze Thm.~\ref{thm:QkLambda1PlusEpsilonCmaxLambda}.
In this context, first note that the set $\C = [-2,2]^n$
is $\lambda$-contractive for every $\lambda \in [0.6,1]$.
Moreover, it is easy to show that the maximal $\lambda$-contractive set is given by 
\begin{equation}
\label{eq:CMaxLambdaExample}
\C_{\max}^{\lambda} = [-c_{\max}^{\lambda},c_{\max}^\lambda]^n \qquad \text{with} \qquad c_{\max}^\lambda :=\frac{1}{1.1-\lambda}
\end{equation}
for this example.
We obviously  have $\C_{\max}^{\lambda}=\C$ for $\lambda = 0.6$ and $C_{\max}^{\lambda} = \X$ for $\lambda =1$.
Now, according to Thm.~\ref{thm:QkLambda1PlusEpsilonCmaxLambda}, for every $\lambda \in [0.6,1]$ and every $\epsilon >0$, there exists a $k \in \Nat$ such that~\eqref{eq:QkLambdaXSubsetQkLambdaC} holds.
Following Thm.~\ref{thm:QkLambda1PlusEpsilonCmaxLambda},
such a $k$ can be found by setting $\D=\X$ and $\delta = \ln(1+\epsilon)$,  selecting $\eta$
as in Cor.~\ref{cor:choiceOfEta}, and choosing $k$ such that~\eqref{eq:kGeg2Cases} holds.
To this end, first note that $\lb{r}_x = 10$, $\ub{r}_x = 10\sqrt{n}$, and $\lb{r}_u = 1$
are such that~\eqref{eq:rUrX} holds.
Moreover, $\alpha = 1.1^n$ satisfies~\eqref{eq:alpha}
and it is straightforward to show that
$$
\begin{array}{c}
\sigma_{\max}(\Phi_n)=\sigma_{\min}(\Phi_n)=\sqrt{\sum_{i=0}^{n-1} 1.1^{2 i}} = \sqrt{\frac{1.21^n -1 }{0.21}}.
\end{array}
$$
Thus, according to~\eqref{eq:choiceOfRho}, $\hat{\varrho}$ evaluates to
$$
\hat{\varrho} = \frac{\lambda^{n-1}}{1.1^n} \min \left\{ 5,\sqrt{\frac{1.21^n -1 }{0.21}} \right\}.
$$
Consequently, a suitable choice for $\eta$ is 
\begin{equation}
\label{eq:etaExample}
\eta = 1 - \frac{\hat{\varrho}}{\ub{r}_x}= 1 - \frac{\lambda^{n-1}}{10 \sqrt{n}\,1.1^n} \min \left\{ 5,\sqrt{\frac{1.21^n -1 }{0.21}} \right\}.
\end{equation}
It remains to choose $k$ satisfying~\eqref{eq:kGeg2Cases}.
Evaluating the distance between $\C$ and $\X$ according to~\eqref{eq:dComputation} results in $d(\C,\X)=\ln(5)$. 
Thus, the smallest $k$ that satisfies~\eqref{eq:kGeg2Cases} for  any $\epsilon \in (0,4)$ can be computed according to
\begin{equation}
\label{eq:kThmExample}
k =n \, \left\lceil \frac{\ln(\ln(1+\epsilon))-\ln(\ln(5))}{\ln(\eta)} \right \rceil
\end{equation}
with $\eta$ as in~\eqref{eq:etaExample}.
Numerical values for 
$n \in \{1,2\}$, $\lambda \in \{0.6,0.8,1.0\}$, and $\epsilon \in \{0.01,0.05,0.1\}$
are listed in Tab.~\ref{tab:kThmExample}.(a).

We next compare the results in Tab.~\ref{tab:kThmExample}.(a) with the smallest $k$ such that~\eqref{eq:QkLambdaXSubsetQkLambdaC} holds.
For this simple example, $\Q_k^\lambda(\C)$ and $\Q_k^\lambda(\X)$ can be stated explicitly as
\begin{equation}
\label{eq:QkLambdaCExample}
\Q_k^\lambda(\C) = [-c_k^\lambda,c_k^\lambda]^n
\,\,\,\,
\text{with}
\,\,\,\,
c^\lambda_k :=  2\,\left(\frac{\lambda}{1.1}\right)^k
+ \frac{1-\left(\frac{\lambda}{1.1}\right)^k}{1.1-\lambda}
\end{equation}
and
\begin{equation}
\label{eq:QkLambdaXExample}
\!\,\,\Q_k^\lambda(\X) = [-x_k^\lambda,x_k^\lambda]^n
\,\,\,\,
\text{with}
\,\,\,\,
x^\lambda_k := 10\, \left(\frac{\lambda}{1.1}\right)^k\!\!\!
+ \frac{1-\left(\frac{\lambda}{1.1}\right)^k}{1.1-\lambda}\!\!\!
\end{equation}
for every $k \in \Nat$ and any $\lambda \in [0.6,1]$.
Now, condition~\eqref{eq:QkLambdaXSubsetQkLambdaC} obviously holds for a given $k\in \Nat$ if (and only if)
\begin{equation}
\label{eq:kConditionExample}
x^\lambda_k \leq (1+\epsilon) \,c_{k}^\lambda.
\end{equation}
Clearly, \eqref{eq:kConditionExample} does not depend on $n$. In other words, the smallest $k$ such that~\eqref{eq:QkLambdaXSubsetQkLambdaC} holds does not change with~$n$ for this example. Apparently, this observation  is not echoed in Tab.~\ref{tab:kThmExample}.(a) (or Eq.~\eqref{eq:kThmExample}), where we clearly have a dependence on $n$.
Now, based on~\eqref{eq:QkLambdaXExample} and \eqref{eq:QkLambdaCExample}, it is easy to prove that the smallest $k$ satisfying~\eqref{eq:kConditionExample} reads
\begin{equation}
\label{eq:kExact}
k =\left\lceil \frac{\ln\left(\left(1.1-\lambda\right) \left(\frac{8}{\epsilon}-2 \right) +1 \right)}{\ln(1.1)-\ln(\lambda)} \right\rceil
\end{equation}
for a given $\lambda \in [0.6,1]$ and $\epsilon \in (0,4]$.
In Tab.~\ref{tab:kThmExample}.(b) numerical values for $k$ as in~\eqref{eq:kExact} are listed for $\lambda$ and $\epsilon$ as above.
Comparing the entries in Tabs.~\ref{tab:kThmExample}.(a) and \ref{tab:kThmExample}.(b), it turns out that
 the values for $k$ computed according to Thm.~\ref{thm:QkLambda1PlusEpsilonCmaxLambda} are valid but conservative.
In fact, the smallest overestimation, which is by a factor of $\frac{54}{47} \approx 1.1489$, occurs for $n=1$, $\lambda=1.0$, and $\epsilon = 0.01$.

\begin{table}[htp]
\caption{Numerical values for $k$ as in~\eqref{eq:kThmExample} and~\eqref{eq:kExact}, respectively,\newline as a function of  $n$, $\lambda$, and $\epsilon$.}
\label{tab:kThmExample}
\vspace{1mm}
\centering
\begin{tabular}{ccrrr}
\multicolumn{5}{c}{(a) values for $k$ as in~\eqref{eq:kThmExample}}\\[1mm]
\toprule
$n$ & $\lambda$ & \multicolumn{3}{c}{$\epsilon$} \\
\cmidrule{3-5}
&  & $0.01$ & $0.05$ & $0.1$ \\ 
\midrule
$2$ &   $0.6$ & $192$ &  $132$  & $106$  \\
$2$ &   $0.8$ & $142$ &   $98$  &  $80$  \\
$2$ &   $1.0$ & $112$  &  $78$   & $64$  \\
\cmidrule{1-2}
$1$ &   any & $54$  &  $37$  &  $30$ \\
\bottomrule
  \end{tabular}\hspace{15mm}
  \begin{tabular}{crrr}
  \multicolumn{4}{c}{(b) values for $k$ as in~\eqref{eq:kExact}}\\[1mm]
\toprule
$\lambda$ & \multicolumn{3}{c}{$\epsilon$} \\
 \cmidrule{2-4}
(any\,$n$) & $0.01$ &  $0.05$  & $0.1$ \\ 
\midrule
    $0.6$ & $10$  & $8$   & $7$   \\
   $0.8$ & $18$  & $13$  & $11$   \\
   $1.0$ & $47$  & $30$ & $23$ \\[5.6mm]
 \bottomrule
  \end{tabular}
\end{table}

Another observation is also interesting.
In Tab.~\ref{tab:kThmExample}.(b),  for fixed $\epsilon$, the values of $k$ increase with increasing $\lambda$. 
In contrast, for $n=2$ and fixed $\epsilon$, the values of $k$ in Tab.~\ref{tab:kThmExample}.(a) decrease with increasing $\lambda$.
In general, for some $\lambda^\ast \leq \lambda$ and $\D \subseteq \X$, we have $\C_{\max}^{\lambda^\ast} \subseteq \C_{\max}^{\lambda}$ and $Q_{k}^{\lambda^\ast}(\D) \subseteq Q_{k}^{\lambda}(\D)$ for every $k \in \Nat$.
In other words, larger values of $\lambda$ imply a larger set $\C_{\max}^\lambda$, slower contraction of $\{Q_{k}^\lambda(\X)\}$,
and faster expansion of $\{Q_k^\lambda(\C)\}$.
Clearly, this observation does not allow a general statement about the dependence of the smallest $k$, such that~\eqref{eq:QkLambdaXSubsetQkLambdaC} holds, on~$\lambda$.
In fact, depending on the example, we may observe a behavior similar or opposite to Tab.~\ref{tab:kThmExample}.(b) (i.e., $k$ increases or decreases with $\lambda$). 
In contrast, the iteration bound $k$ considered in Thm.~\ref{thm:QkLambda1PlusEpsilonCmaxLambda} will always decrease with $\lambda$ (as apparent from Tab.~\ref{tab:kThmExample}.(a)).
This behavior can be explained with regard to Cor.~\ref{cor:choiceOfEta} and Lem.~\ref{lem:appropriateKCD}.
Clearly, for larger $\lambda$, $\hat{\varrho}$ as in~\eqref{eq:choiceOfRho} will be larger, which results in smaller $\eta = 1 - \frac{\hat{\varrho}}{\ub{r}_x}$ and finally smaller $k$ satisfying~\eqref{eq:kGeg2Cases}.
While this behavior may be conservative (as it is for this example), it is required to prove Thm.~\ref{thm:muC1MaxSubsetCLambdaMax} and the underlying Lem.~\ref{lem:muLambdaK}. Indeed, the strategy to handle case (ii) in the proof of Lem.~\ref{lem:muLambdaK} builds on the fact that the computed $\eta$ for some $\lambda \in (\lambda^\ast,1]$ is smaller than the one for $\lambda=\lambda^\ast$.

\subsection{Discussion of Theorem~\ref{thm:muC1MaxSubsetCLambdaMax}}
\label{subsec:ExampleThm12}

Theorem~\ref{thm:muC1MaxSubsetCLambdaMax} allows  $\lambda$ to be chosen such that \eqref{eq:muC1MaxSubsetCLambdaMax} is guaranteed to hold for a given $\mu$. In the following, we compare the  smallest value of $\lambda$ such that~\eqref{eq:muC1MaxSubsetCLambdaMax} holds with the value that is obtained using Lem.~\ref{lem:muLambdaK} for the example from Sect.~\ref{subsec:ExThm9} and $n=1$. 


Assume we want to  satisfy~\eqref{eq:muC1MaxSubsetCLambdaMax} for $\mu=\frac{5}{6}$. Before applying Thm.~\ref{thm:muC1MaxSubsetCLambdaMax} (and Lem.~\ref{lem:muLambdaK}), first note that the maximal controlled invariant set is
$\C_{\max}^1 = [-10, 10]$ (according to~\eqref{eq:CMaxLambdaExample} with $\lambda=1$).
Obviously, $\mu \,\C_{\max}^1 \subseteq \C_{\max}^\lambda$ requires
\begin{equation}
\label{eq:muLambdaCondition}
10 \,\mu = \frac{50}{6} \leq \frac{1}{1.1-\lambda} \quad \text{or, equivalently,} \quad \frac{49}{50}\leq \lambda.
\end{equation}
Thus, $\lambda^\ast = 0.98$ is the smallest choice for $\lambda$ such that~\eqref{eq:muC1MaxSubsetCLambdaMax} holds. 

The computation of a suitable $\lambda$ according to Thm.~\ref{thm:muC1MaxSubsetCLambdaMax} (and Lem.~\ref{lem:muLambdaK}) involves finding $\lambda$ and $k$
 such that~\eqref{eq:QkLambdaXSubsetQkLambdaC} and~\eqref{eq:muLambdaK} hold for the choice
$$
\epsilon = \frac{1-\mu}{2\,\mu}=\frac{1}{10}=0.1.
$$
In this context, Thm.~\ref{thm:muC1MaxSubsetCLambdaMax} and Lem.~\ref{lem:muLambdaK} require the knowledge of a $\lambda^\ast$-contractive set $\C$. We again consider the set $\C=[-2,2]$ from Sect.~\ref{subsec:ExThm9}, which is $\lambda$-contractive for every $\lambda \in [0.6,1]$.
We use $\lambda^\ast=0.98$ from above as an initial guess for the computation of a suitable $\lambda$ corresponding to the proof of Lem.~\ref{lem:muLambdaK}. In other words, we first analyze whether the presented procedure is capable of identifying whether $\lambda^\ast$ is suitable.
Clearly, the smallest $k$ satisfying~\eqref{eq:kGeg2Cases} for $\D=\X$ and $\delta = \ln(1+\epsilon)$
can be computed analogously to Sect.~\ref{subsec:ExThm9}. 
Hence, evaluating~\eqref{eq:kThmExample} for $n=1$, $\epsilon=0.1$, and $\eta$ from \eqref{eq:etaExample} yields $k=30$ (as itemized in Tab.~\ref{tab:kThmExample}.(a)).
We obtain
$$1+\mu \approx 1.8333 > 1.0910 \approx 2 \cdot 0.98^{30} = 2 \, (\lambda^{\ast})^k,$$
i.e.,~\eqref{eq:muLambdaK} does not hold for the choice $\lambda=\lambda^\ast$ and we have to address case (ii) in the proof of Lem.~\ref{lem:muLambdaK}. Consequently, updating $\lambda$ according to
~\eqref{eq:lambdaUpdate} yields
$\lambda \approx 0.9971$.
Following the argumentation in the proof of Lem.~\ref{lem:muLambdaK},
the updated $\lambda$ and $k=30$ are such that~\eqref{eq:QkLambdaXSubsetQkLambdaC} and~\eqref{eq:muLambdaK} hold. 
Thus, according to Thm.~\ref{thm:muC1MaxSubsetCLambdaMax}, the updated $\lambda$ is such that~\eqref{eq:muC1MaxSubsetCLambdaMax} holds. 
The computed $\lambda$ is conservative in the sense that
$$
\frac{\lambda-\lambda^\ast}{1-\lambda^\ast} = 0.8552 = 85.52 \%
$$
of the ``suitable interval'' $[0.98,1)$ is not identified as being suitable.
However, the result can also be interpreted in a different way.
To this end, we compute $\lambda$-contractive sets $\T$ that accurately approximate $\C_{\max}^1$ according to the two strategies in Rem.~\ref{rem:ComputationLambdaSetPractice}. Using the first strategy, we obtain $\T=\Q_{30}^{\lambda}(\C)$ based on the iteration bound $k=30$. The second strategy leads to an earlier termination after $k^\ast=23$ iterations (according to~\eqref{eq:kExact}).
This observation is interesting, since any choice $\lambda \in [0.98,0.9971]$ requires at least 21 iterations to satisfy~\eqref{eq:QkLambdaXSubsetQkLambdaC} with $\epsilon = 0.1$.
In other words, the conservatism in the choice of $\lambda$ only slightly influences the earliest satisfaction of~\eqref{eq:QkLambdaXSubsetQkLambdaC}.


 \subsection{Discussion of Theorem~\ref{thm:dQnCDLeqEtaDCD}}
\label{subsec:ExampleThm2}
  
Theorems~\ref{thm:QkLambda1PlusEpsilonCmaxLambda} and \ref{thm:muC1MaxSubsetCLambdaMax} both build on the contraction property in Thm.~\ref{thm:dQnCDLeqEtaDCD}.   
It thus makes sense to discuss Thm.~\ref{thm:dQnCDLeqEtaDCD} in more detail.

First, it is important to note that the contraction property in Thm.~\ref{thm:dQnCDLeqEtaDCD} only applies to the mapping $\Q_n^\lambda(\C)$, where $n$ refers to the state space dimension. Initially, this seems counter-intuitive and one would expect a contraction after every step $k$. In fact, the example discussed in Sect.~\ref{subsec:ExThm9} (and Sect.~\ref{subsec:ExampleThm12}) shows such a behavior. There exist, however, situations where a contraction indeed only appears every $n$ steps. In this context, consider system~\eqref{eq:linSys} with
$$
A=\begin{pmatrix}
\textcolor{white}{+}0 & 1 \\ -1 & 0 
\end{pmatrix} \quad \text{and} \quad
B=\begin{pmatrix}
0 \\ 1 
\end{pmatrix} 
$$
and constraints $\X=[-5,5]^2$ and $\U=[-1,1]$ (which is taken from \cite[Sect. IV-B]{SchulzeDarup2014_CDC}). We will show that, for the choice $\C=[-1,1] \times [-1,1]$ and $\D=[-2,2] \times [1,1]$,
we obtain
\begin{align}
\nonumber
d(\Q_{2j+1}^\lambda(\C),\Q_{2j+1}^\lambda(\D))&=d(\Q_{2j}^\lambda(\C),\Q_{2j}^\lambda(\D))\\
\label{eq:dQCD2j}
&=  \ln \left(\frac{\lambda^{2j}}{\sum_{i=0}^j \lambda^{2i}}+1\right)
\end{align}
for every $\lambda \in (0,1]$ and every $j \in \Nat_{[0,3]}$ (for $j> 3$, i.e., $k > 7$, the state constraints $\X$ may, depending of the choice of $\lambda$, affect the shapes of $\Q_{k}^\lambda(\C)$ and $\Q_{k}^\lambda(\D)$ so that~\eqref{eq:dQCD2j} may no longer hold).
Now, according to~\eqref{eq:dQCD2j}  for $j=0$, we find 
$$
d(\Q_{1}^\lambda(\C),\Q_{1}^\lambda(\D))=d(\Q_{0}^\lambda(\C),\Q_{0}^\lambda(\D))=d(\C,\D)=\ln(2).
$$
In other words, \eqref{eq:weakerRelation} holds for $k=1<n$ with equality (in agreement with Cor.~\ref{cor:dQkLambdaLeqD}) but there is no contraction in terms of the distance between the sets after one iteration.
For $k=n=2$, relation~\eqref{eq:dQnCDLeqEtaDCD} can, however, be easily satisfied for the choice  $\eta =1-\frac{\lambda}{\sqrt{50}}$. This follows from Cor.~\ref{cor:choiceOfEta} with
 $\lb{r}_x=5$, $\ub{r}_x=\sqrt{50}$, $\lb{r}_u=1$, $\alpha=1$, 
$\Phi_n=I_2$, $\sigma_{\max}(\Phi_n)=\sigma_{\min}(\Phi_n)=1$, and 
$\hat{\varrho}= \lambda \min \left\{ \frac{\lb{r}_x}{2},\, \lb{r_u}  \right\} = \lambda.$
To prove~\eqref{eq:dQCD2j}, finally note that $\Q_1^\lambda(\T)$ 
evaluates to
\begin{equation}
\label{eq:Q1TExample}
\Q_1^\lambda(\T) = [-\lambda\,\tau_2 - 1,\lambda\,\tau_2 + 1] \times [-\lambda\,\tau_1,\lambda\,\tau_1]
\end{equation}
for any set 
$\T=[-\tau_1,\tau_1] \times [-\tau_2,\tau_2] $
with $\tau_1 \in (0,5]$ and $\tau_2 \in (0,4]$.
Equation~\eqref{eq:Q1TExample} allows to compute $\Q_{k}^\lambda(\C)=\Q_{1}^\lambda(\Q_{k-1}^\lambda(\C))$ and $\Q_{k}^\lambda(\D)=\Q_{1}^\lambda(\Q_{k-1}^\lambda(\D))$ for every $k\in \Nat_{[1,7]}$ (for $k>7$, the conditions on $\tau_1$ and $\tau_2$ may be violated for $\T=\Q_{k-1}(\C)$ or $\T=\Q_{k-1}(\D)$).
Afterwards, \eqref{eq:dComputation} can be used to evaluate the distances $d(\Q_{k}^\lambda(\C),\Q_{k}^\lambda(\D))$. 
Identifying relation~\eqref{eq:dQCD2j} is then straightforward.

Another important limitation of Thm.~\ref{thm:dQnCDLeqEtaDCD} is that the pair $(A,B)$ has to be controllable (see Assum.~\ref{assum:AB}). Clearly, it would be desirable to extend the contraction property to systems that are ``only'' stabilizable.
However, a simple extension is not possible as the following example shows.
Consider system~\eqref{eq:linSys} with
system matrices $A=0.8 $ and $B=0$ and constraints $\X=[-5,5]$ and $\U=[-1,1]$. Note that the pair $(A,B)$ is stabilizable but not controllable. 
We show in the following that, for the choice $\C=[-1,1]$ and $\D=[-2,2]$, we obtain
\begin{equation}
\label{eq:dQkEqualD0}
d(\Q_k^\lambda(\C),\Q_k^\lambda(\D)) = d(\C,\D) = 1
\end{equation}
for every $k \in \Nat_{[0,4]}$ and every $\lambda \in (0,1]$.
Obviously, since~\eqref{eq:dQkEqualD0} applies for $k=n=1$, relation~\eqref{eq:dQnCDLeqEtaDCD}  cannot hold with $\eta<1$ for some stabilizable systems.
In other words, Thm.~\ref{thm:dQnCDLeqEtaDCD} can in this form not be extended to those systems.
 Corollary~\ref{cor:dQkLambdaLeqD}, which does not require controllability of $(A,B)$, is however consistent with~\eqref{eq:dQkEqualD0}.
To show~\eqref{eq:dQkEqualD0}, note that $\Q_1^\lambda(\T)$ 
evaluates to  $\Q_1^\lambda(\T)=[-\frac{\lambda\,\tau}{0.8},\frac{\lambda\,\tau}{0.8}]$ for any set $\T=[-\tau,\tau] $ with $\tau \in (0,4]$.
Computing $\Q_k^\lambda(\C)$ and $\Q_k^\lambda(\D))$ accordingly and evaluating $d(\Q_k^\lambda(\C),\Q_k^\lambda(\D)$ as in~\eqref{eq:dComputation} leads to~\eqref{eq:dQkEqualD0}. 
We finally note that~\eqref{eq:dQkEqualD0} holds for every $k\in \Nat$ in case of $\lambda \in (0,0.8]$. For $\lambda \in (0.8,1]$, however,  $\Q_k^\lambda(\C)$ and $\Q_k^\lambda(\D)$ converge to the state constraints $\X$, which eventually results in violation of~\eqref{eq:dQkEqualD0}.

\section{Conclusion}

The paper presented two interesting  results related to the computation of $\lambda$-contractive sets for linear constrained systems. 
First, we showed that it is possible to a priori compute a number of iterations $k$ that is sufficient to approximate the largest $\lambda$-contractive set $\C_{\max}^\lambda$ with a given precision $\epsilon$ using the sequence~\eqref{eq:QkLambdaSequence}. Formally, this result is summarized in Thm.~\ref{thm:QkLambda1PlusEpsilonCmaxLambda}.
Second, we showed in Thm.~\ref{thm:muC1MaxSubsetCLambdaMax} how to compute a suitable $\lambda$ such that the associated maximal $\lambda$-contractive set is guaranteed to approximate the maximal controlled invariant set $\C_{\max}^1$  with a given accuracy.
The statements in Thms.~\ref{thm:QkLambda1PlusEpsilonCmaxLambda} and~\ref{thm:muC1MaxSubsetCLambdaMax} were illustrated with an example.
As one might expect, we found that the computed iteration bound $k$ and the provided choice for $\lambda$ are valid but conservative.
Nevertheless, the procedure for a suitable choice of $\lambda$  guaranteeing~\eqref{eq:muC1MaxSubsetCLambdaMax} might be useful for practical computations of $\lambda$-contractive sets since the conservatism in $\lambda$ only slightly influences the termination of step-set based approximations of $\C_{\max}^\lambda$ (see the example in Sect.~\ref{subsec:ExampleThm12}).

Theorems~\ref{thm:QkLambda1PlusEpsilonCmaxLambda} and~\ref{thm:muC1MaxSubsetCLambdaMax} both build on the contraction property summarized in Thm.~\ref{thm:dQnCDLeqEtaDCD} and 
the iteration bound introduced in Lem.~\ref{lem:appropriateKCD}.
The statements in Thm.~\ref{thm:dQnCDLeqEtaDCD} and Lem.~\ref{lem:appropriateKCD}
require the pair $(A,B)$ to be controllable (see Assum.~\ref{assum:AB}) and this restriction is passed on to Thms.~\ref{thm:QkLambda1PlusEpsilonCmaxLambda} and~\ref{thm:muC1MaxSubsetCLambdaMax}. Clearly, it would be desirable to extend all statements to systems that are ``only'' stabilizable.
It was, however, shown that there exist stabilizable systems for which the statement in Thm.~\ref{thm:dQnCDLeqEtaDCD} does not apply (see the latter example in Sect.~\ref{subsec:ExampleThm2}).
Nevertheless, there is no fundamental argument against the extension of Thms.~\ref{thm:QkLambda1PlusEpsilonCmaxLambda} and~\ref{thm:muC1MaxSubsetCLambdaMax} to stabilizable systems. Consequently, future work has to address these non-trivial extensions in order to complete the theory.

\section*{Acknowledgment}

We thank the anonymous reviewers for their helpful comments and suggestions. Financial support by the German Research Foundation
(DFG) through the grant SCHU 2094/1-1 is gratefully
acknowledged.

\appendix

\section{Additional proof}

\begin{proof}[Proof of Lem.~\ref{lem:dCDepsilon}]
The proof consists of three parts addressing (i) ``$\Longrightarrow$''  in~\eqref{eq:dImplications},  (ii) ``$\Longleftarrow$''  in~\eqref{eq:dImplications}, and (iii) relation~\eqref{eq:dComputation}. As a preparation, note that $\rho(\xi,\C)\leq \rho(\xi,\D)$ for every $\xi \in \S$ due to $\C \subseteq \D$.

Part (i).  Having $\D \subseteq \exp(\delta) \,\C$ implies $\rho(\xi,\D) \leq \rho(\xi,\exp(\delta)\,\C)$ for every $\xi \in \S$. We further obtain $\rho(\xi,\exp(\delta)\,\C)= \exp(\delta)\,\rho(\xi,\C)$ for every $\xi \in \S$ by definition of $\rho(\cdot,\cdot)$. Thus 
$$d(\C,\D) = \sup_{\xi \in \S} \, \ln \left(\frac{\rho(\xi,\D)}{\rho(\xi,\C)} \right)  \leq \ln(\exp(\delta))=\delta.
$$ 

Part (ii). Assume $d(\C,\D) \leq \delta$ but $\D \nsubseteq \exp(\delta) \,\C$. Then, there exists an $x \neq 0$ such that $x\in \D$ but $x \notin \exp(\delta)\,\C$. Define $\xi^\ast := \frac{x}{\|x\|_2}$ and note $\xi^\ast \in \S$. Clearly, $\rho(\xi^\ast,\D)>\rho(\xi^\ast,\exp(\delta) \,\C)= \exp(\delta)\,\rho(\xi^\ast,\C)$. This, however, contradicts $d(\C,\D) \leq \delta$ since 
$$\ln \left(\frac{\rho(\xi^\ast,\D)}{\rho(\xi^\ast,\C)}\right) > \ln(\exp(\delta))=\delta.
$$

Part (iii). Let $\mu^\ast := \arg\min_\mu \ln(\mu)$ s.t. $\D \subseteq \mu \,\C$ and define $\delta^\ast := \ln(\mu^\ast)$. Then, we have $\D \subseteq  \exp(\delta^\ast) \, \C$ and consequently $d(\C,\D) \leq \delta^\ast$ according to part (i) of the proof. Now, assume $d(\C,\D)=\delta<\delta^\ast$. Then, $\D \subseteq \exp(\delta) \,\C$ according to part (ii). This, however, contradicts $\mu^\ast$ being the optimizer of~\eqref{eq:dComputation} since $\mu\!=\!\exp(\delta)\!<\!\mu^\ast$. Thus, $d(\C,\D)\!=\!\delta^\ast$ in accordance with~\eqref{eq:dComputation}.
\end{proof}

\end{document}